\documentclass[11pt]{article}
\usepackage{enumerate}
\usepackage{amssymb,amsmath,amsfonts,amsthm,mathrsfs}
\usepackage[colorlinks=true, pdfstartview=FitV, linkcolor=blue, citecolor=blue, urlcolor=blue]{hyperref}
\usepackage{graphicx,color}
\usepackage{cite}
\usepackage{indentfirst}

\allowdisplaybreaks

\usepackage{latexsym}
\usepackage{fancyhdr}
\allowdisplaybreaks \allowdisplaybreaks[4]
 \usepackage{comment}
\def\rn{{\mathbb R^n}}

\def\nm{{\mathbb (\rn)^m}}

\def\n{{\mathbb N}}

\def\hn{{\mathbb H^n}}

\newtheorem{thm}{Theorem}[section]

\newtheorem{lem}[thm]{Lemma}

\theoremstyle{definition}
\newtheorem{defin}[thm]{Definition}
\newtheorem{rem}[thm]{Remark}

\numberwithin{equation}{section}

\begin{document}


\baselineskip=17pt

\renewcommand{\thefootnote}{\fnsymbol {footnote}}

\title{\bf\Large Boundedness of the multilinear integral operators on Heisenberg group
\footnotetext{
\newline\indent\hspace{1mm} {\it 2020 Mathematics Subject Classification}: Primary 43A15; Secondary 42b35 26D10.}}
\date{}
\author{Xiang Li, Xi Cen, Zunwei Fu and Zhongci Hang\footnote{Corresponding author}}
\maketitle



\begin{abstract}
In this paper, based on Stein-Weiss Lemma, we study the sharp constants for multilinear integral operators with the nonnegative kernels on Heisenberg group Lebesgue space, weighted Lebesgue space and weighted Morrey space. Using the integral, we can easily calculate the sharp constants for multilinear Hilbert operator, the multilinear Hardy-Littlewood-P$\acute{o}lya$ operator  and the multilinear Hardy operator on Heisenberg group Lebesgue space. In addition, we also consider the boundedness for multilinear Calder$\acute{o}$n-Zygmund operator on $A_p$ weighted Morrey space. 
\end{abstract}

\newtheorem{theorem}{Theorem}[section]
\newtheorem{preliminaries}{Preliminaries}[section]
\newtheorem{definition}{Definition}[section]
\newtheorem{main result}{Main Result}[section]
\newtheorem{lemma}{Lemma}[section]
\newtheorem{proposition}{Proposition}[section]
\newtheorem{corollary}{Corollary}[section]
\newtheorem{remark}{Remark}[section]

\section[Introduction]{Introduction}
For a linear integral operator with the nonnegative kernel which satisfy some homogeneity and the rotational invariance, Beckner \cite{Beckner} obtained the following result which is also called Stein-Weiss lemma.

\noindent{\bf Theorem A.}
Suppose that $K$ is a nonnegative kernel on $\mathbb{R}^n\times\mathbb{R}^n$ that satisfies continuous on any domain excluded the point $(0,0)$, homogeneous of degree $-n$, $K(\delta u,\delta v)=\delta^{-n}K(u,v)$ and $K(Ru,Rv)=K(u,v)$ for any $ R \in \mathrm SO(n)$. Then $T$ is a from $L^p(\mathbb{R}^n)$ to $L^p(\mathbb{R}^n)$ bounded integral operator defined by
$$
 T f(x)=\int_{\mathbb{R}^n} K(x, y) f(y) d y,
$$
with
$$
\|T f\|_{L^p(\mathbb{R}^n)} \leq C\|f\|_{L^p(\mathbb{R}^n)},
$$
holds for $1<p<\infty$, the optimal constant is given by
$$
C=\int_{\mathbb{R}^n} K(x, e_1)|x|^{-\frac{n}{p^{\prime}}} d x,
$$
where $e_1=(1,0,\ldots,0)$ is a unit vector in the first coordinate direction in $\mathbb{R}^n$ and $\mathrm SO(n)$ denotes a set of the rotation transformation on $\mathbb{R}^n$.

In \cite{DW}, Wu and Yan extended Theroem \textbf{A.} to  multilinear setting and gave the sharp constant accordingly.  Now, we recall their main results as follows.

\noindent{\bf Theorem B.}
Suppose that $K$ is a nonnegative kernel defined on $\mathbb{R}^{nm}$ and satisfies  two conditions simultaneously.

(1) $K$ is homogeneous of degree $-nm$, i.e. and
$$
K( {\delta x,\delta {x_1}, \ldots ,\delta {x_m}} ) = {\delta ^{ - nm}}K( {x,{x_1}, \ldots ,{x_m}} )
$$
holds for any $\delta>0$;

(2) For any $R\in\mathrm SO(n)$, it follows that
$$
K(R x, R x_1, \ldots, R x_m)=K(x, x_1, \ldots, x_m).
$$
Then they defined the multilinear integral operator with kernel $K$ such that
$$
{T_K}( {{f_1}, \ldots ,{f_m}} )(x) = \int_{{\mathbb{R}^{n m}}} {K( {x,{x_1}, \ldots ,{x_m}} )\prod\limits_{i = 1}^m {{f_i}} ( {{x_i}} )d{x_1} \cdots d{x_m}},
$$
where $T_K$ maps $L^{p_1}(\mathbb{R}^n)\times\cdots\times L^{p_m}(\mathbb{R}^n)$ to $L^p(\mathbb{R}^n)$ if and only if

$(a)$

$$
\frac{1}{p_1} +\frac{1}{p_2}+\cdots+\frac{1}{p_m}=\frac{1}{p}
$$
and

$(b)$

$$
\int_{\mathbb{R}^{nm}} K(e_1, x_1, \ldots, x_m) \prod_{i=1}^m|x_i|^{-\frac{n}{p_i}} d x_1 \cdots d x_m<\infty,
$$
where  $1 \leq p_i , p \leq \infty$ for $i=1,\ldots,m$, $e_1=(1,0,\ldots,0)$ is a unit vector in the first coordinate direction in $\mathbb{R}^n$ and $\mathrm SO(n)$ denotes a set of the rotation transformation on $\mathbb{R}^n$. The operator $T_k$ is bounded and they have obtained the norm of the operator $T_K$, i.e.,
$$
\|T_K\|_{L^{p_1}(\mathbb{R}^n) \times \cdots \times L^{p_m}(\mathbb{R}^n) \rightarrow L^p(\mathbb{R}^n)}=\int_{\mathbb{R}^{nm}} K(e_1, x_1, \ldots, x_m) \prod_{i=1}^k|x_i|^{-\frac{n}{p_i}} d x_1 \cdots d x_m
$$
holds.

Naturally, it will be a very interesting problem to ask whether we can establish the multilinear Stein-Weiss lemma in Heisenberg group. In this paper, we will give a positive answer. Sharp constants of some function space theory have attracted much attention of analysts for more than a century, which plays  important roles in several branches of mathematics, such as representation theory, harmonic analysis, several complex variables, partial differential equations and quantum mechanics, for more details can see \cite{Stein}. As a special form of linear integral operator, the sharp bound of  Hardy operator on Heisenberg group Lebesgue space has been obtained by Wu and Fu in \cite{Wu}, we will give several special cases of multilinearity.
Another of our main result is  obtain the boundedness of mutilinear Calder$\acute{o}$n-Zygmund operator on Heisenberg group $A_p$ weighted Morrey space.  

Now, we recall  some basic knowledge about Heisenberg group. 

The Heisenberg group $\mathbb{H}^n$ is non-commutative nilpotent Lie group, with the underlying manifold $\mathbb{R}^{2n+1}$ and the group law.

Let
$$
x=(x_1,\ldots,x_{2n},x_{2n+1}),y=(y_1,\ldots,y_{2n},y_{2n+1}),
$$
then
$$
x \circ y=\left(x_1+y_1, \ldots, x_{2 n}+y_{2 n}, x_{2 n+1}+y_{2 n+1}+2 \sum_{j=1}^n(y_j x_{n+j}-x_j y_{n+j})\right).
$$
Note that  Heisenberg group $\mathbb{H}^n$ is a homogeneous group with dilations
$$
\delta_r(x_1, x_2, \ldots, x_{2 n}, x_{2 n+1})=(r x_1, r x_2, \ldots, r x_{2 n}, r^2 x_{2 n+1}), \quad r>0.
$$

The Haar measure on $\mathbb{H}^n$ coincides with the usual Lebesgue measure on $\mathbb{R}^{2n+1}$. We denote any measurable set $E \subset \mathbb{H}^n$ by $|E|$, then
$$
|\delta_r(E)|=r^Q|E|, d(\delta_r x)=r_Q d x,
$$
where $Q=2n+2$ is called the homogeneous dimension of $\mathbb{H}^n$.

The Heisenberg distance derived from the norm
$$
|x|_h=\left[\left(\sum_{i=1}^{2 n} x_i^2\right)^2+x_{2 n+1}^2\right]^{1 / 4},
$$
where $x=(x_1,x_2,\ldots,x_{2n},x_{2n+1})$ is given by
$$
d(p, q)=d(q^{-1} p, 0)=|q^{-1} p|_h.
$$
This distance $d$ is left-invariant in the sense that $d(p,q)$ remains unchanged when $p$ and $q$ are both left-translated by some fixed vector on $\mathbb{H}^n$. Besides, $d$ satisfies the triangular inequality defined by \cite{Kor}
$$
d(p, q) \leq d(p, x)+d(x, q), \quad p, x, q \in \mathbb{H}^n.
$$
For $r>0$ and $x\in\mathbb{H}^n$, the ball and sphere with center $x$ and radius $r$ on $\mathbb{H}^n$ are given by
$$
B(x, r)=\{y \in \mathbb{H}^n: d(x, y)<r\}
$$
and
$$
S(x, r)=\{y \in \mathbb{H}^n: d(x, y)=r\}.
$$
Then we have
$$
|B(x, r)|=|B(0, r)|=\Omega_Q r^Q,
$$
where
$$
\Omega_Q=\frac{2 \pi^{n+\frac{1}{2}} \Gamma(n / 2)}{(n+1) \Gamma(n) \Gamma((n+1) / 2)}
$$
 denote the volume of the unit ball $B(0,1)$ on $\mathbb{H}^n$  that is $\omega_Q=Q \Omega_Q$ (see \cite{CT}). For more details about Heisenberg group can be refer to \cite{GB} and \cite{ST}.

Next, we give the definition of multilinear integral operator with the nonnegative kernel $K$ on Heisenberg group.
\begin{defin}\label{mian_1}
Suppose that $K$ is a nonnegative kernel defined on $\mathbb{H}^{nm}$ and satisfies the homogeneous degree $-nm$,
$$K(|\delta|_h x,|\delta|_h y_1,\ldots,|\delta|_h y_m)=|\delta|^{-nm}_h K(x,y_1,\ldots, y_m)
$$
and$$
K(Rx,Ry_1,\ldots,Ry_m)=K(x,y_1,\ldots,y_m)
$$
for any $R\in SO(n)$,then $K$ denotes the kernel of
\begin{equation}\label{main_2}
H(f_1,\ldots,f_m)(x)=\int_{\mathbb{H}^{nm}}K(x,y_1,\ldots,y_m)\prod^m_{j=1}f_j(y_j)d y_1\cdots d y_m.
\end{equation}
\end{defin}
\section{Sharp constants for multilinear integral operators on Heisenberg group}
In this section, we begin to study multilinear integral operators on Heisenberg group $\mathbb{H}^{n}$. Now we give our main result as follows.

\begin{thm}\label{main_6}
For  operators $H(f_1,\ldots ,f_m)(x)$, which maps $L^p_1(\mathbb{H}^n)\times\cdots\times L^p_m(\mathbb{H}^n)$ to $L^p(\mathbb{H}^n)$. That is,
\begin{equation}
{\left\| H (f_1,\ldots,f_m)\right\|_{{L^{{p_1}}}({\hn}) \times  \cdots  \times {L^{{p_m}}}({\hn}) \to {L^p}({\hn})}} \le {C_m},
\end{equation}
where $$
\frac{1}{p_1}+\frac{1}{p_2}+\cdots+\frac{1}{p_m} =\frac{1}{p}
$$
and $1\leq p_j,p\leq\infty$ for all $j=1,\ldots,m$, the optimal constant is
\begin{equation}
C_m=\int_{\mathbb{H}^{nm}}K(e_1,y_1,\ldots ,y_m)\prod_{i=1}^m |y|_h^{-\frac{Q}{p_i}}d y_1 \cdots d y_m.
\end{equation}
Moreover, if operator $H(f_1,\ldots,f_m)(x)$ is bounded, we can obtain the norm of the operator $H(f_1,\ldots,$
$f_m)(x)$. That is,
\begin{equation}
 {\| H (f_1,\dots,f_m)\|_{{L^{{p_1}}}({\hn}) \times  \cdots  \times {L^{{p_m}}}({\hn}) \to {L^p}({\hn})}} = {C_m}.
\end{equation}
\end{thm}
\begin{thm}\label{61}
	For  operator $H(f_1,\ldots ,f_m)(x)$, which maps $L^{p_1}(|x|_h^{\frac{\alpha_1 p_1}{p}} d x) \times \cdots \times L^{p_m}(|x|_h^{\frac{\alpha_m p_m}{p}} d x)$ to $L^p(|x|_h^{\alpha}dx)$. That is,
	\begin{equation}
		\| H (f_1,\ldots,f_m)\|_{L^{p_1}(|x|_h^{\frac{\alpha_1 p_1}{p}} d x) \times \cdots \times L^{p_m}(|x|_h^{\frac{\alpha_m p_m}{p}} d x)\rightarrow L^p(|x|_h^{\alpha}dx)} \leq {D_m},
	\end{equation}
	where
	$$
	\frac{1}{p_1}+\frac{1}{p_2}+\cdots+\frac{1}{p_m} =\frac{1}{p},\quad\alpha_1+\alpha_2+\cdots+\alpha_m=\alpha
	$$
	and $1\leq p_j$, $p\leq\infty$, $\alpha_j<pQ(1-1/p_j)$ for all $j=1,\ldots,m$, then the sharp constant is
	\begin{equation}
		D_m=\int_{\mathbb{H}^{n m}} K(e_1,y_1,\ldots,y_m)\prod_{j=1}^m|y_j|_h^{-\frac{Q}{p_j}-\frac{\alpha_j}{p}}dy_1 \cdots dy_m.
	\end{equation}
	Moreover, if operator $H(f_1,\ldots,f_m)(x)$ is bounded, we can obtain the norm of the operator $H(f_1,\ldots,$
	$f_m)(x)$. That is
	\begin{equation}
		{\| H (f_1,\dots,f_m)\|_{{L^{{p_1}}}({\hn}) \times  \cdots  \times {L^{{p_m}}}({\hn}) \to {L^p}({\hn})}} = {D_m}.
		\end{equation}
\end{thm}
In order to prove Theorem \ref{main_6} and Theorem \ref{61}, we first need to give the following lemmas.
\begin{lemma}\label{main_10}
Let $1\leq p<\infty$, if $t\in\mathbb{H}^n$ and $f\in L^p(\mathbb{H}^n)$, then we have
\begin{equation}
\|f(\delta_{|t|_h}\cdot)\|_{L^p(\mathbb{H}^n)}=|t|_h^{-\frac{Q}{p}}\|f\|_{L^p(\mathbb{H}^n)}.
\end{equation}
\end{lemma}
\begin{proof}
$$
\begin{aligned}
\|f(\delta_{|t|_h}\cdot)\|_{L^p(\mathbb{H}^n)}&=\left(\int_{\mathbb{H}^n}|f(\delta_{|t|_h} x)|^p dx\right)^\frac{1}{p}\\
&=\left(\int_{\mathbb{H}^n}|f(x)|^p|t|_h^{-Q}dx\right)^\frac{1}{p}\\
&=|t|_h^{-\frac{Q}{p}}\left(\int_{\mathbb{H}^n}|f(x)|^p dx\right)^\frac{1}{p}\\
&=|t|_h^{-\frac{Q}{p}}\|f\|_{L^p(\mathbb{H}^n)}.
\end{aligned}
$$
Lemma \ref{main_10} is thus proved.
\end{proof}
\begin{lemma}\label{main_11}
 Let $K(e_1,y_1,\ldots,y_m)$, where $e_1=(1,0,\ldots,0)$, then we have
$$
K(e_1,y_1,\ldots,y_m)=K(x,y_1,\ldots,y_m).
$$
\end{lemma}
When $\alpha_i=0$, the sharp constant on Heisenberg group Lebesgue space will be easy to get. So we only prove Theorem \ref{61}.
\begin{proof}[Proof of Theorem \ref{61}]

Set
$$
g_j(x)=\frac{1}{\omega_Q} \int_{|\xi_j|_h=1} f_j(\delta_{|x|_h} \xi_j) \mathrm{d} \xi_j, \quad x \in \mathbb{H}^n,
$$
then $g_j(j=1,\ldots,m)$ are radial functions. By change of variables, we have
$$
\begin{aligned}
H(g_{f_1},\ldots ,g_{f_m})(x)&= \int_{\mathbb{H}^{n m}} K\left(x,y_1, \ldots, y_m\right) g_1\left(y_1\right) \ldots g_m( y_m) d y_1 \cdots d y_m  \\
&=\int_{\mathbb{H}^{n m}} K\left(x,y_1, \ldots, y_m\right) \prod_{j=1}^m\left(\frac{1}{\omega_Q} \int_{|\xi_j|=1} f_j(\delta_{|y|_h} \xi_j) d \xi_j\right) d y_1 \cdots d y_m\\
&=\frac{1}{\omega_Q^m}\int_{\mathbb{H}^{n m}} K(x,y_1,\ldots,y_m ) \prod_{j=1}^m\left(\int_{|\xi_j|_h=|y_j|_h} f_j(z_j ) |y_j|_h^{-Q} d z_j \right) d y_1 \cdots d y_m\\
&=\frac{1}{\omega_Q^m}\int_{\mathbb{H}^{n m}} K(x,y_1,\ldots,y_m)\prod_{j=1}^m\left(\int_{|y_j|_h=|z_j|_h}|y_j|_h^{-Q}d y_j \right)\prod_{j=1}^m f_j(z_j) d z_i\cdots d z_m\\
&=\int_{\mathbb{H}^{n m}} K(x,y_1,\ldots,y_m) \prod_{j=1}^m f_j(z_j)d z_i \cdots d z_m\\
&=H(f_1,\ldots,f_m)(x).
\end{aligned}
$$
Using H\"{o}lder's inequality, we have
$$
\begin{aligned}
\|g_j\|_{L^{p_i}(|x|_h^{\frac{p_i\alpha_i}{p}}dx)} & =\frac{1}{\omega_Q}\left(\int_{\mathbb{H}^n}\left|\int_{|\xi_j|_h=1} f(\delta_{|x|_h} \xi_j)d \xi_j\right|^{p_i}|x|_h^{\frac{p_i\alpha_i}{p}} \mathrm{d} x\right)^{1 / {p_i}} \\
& \leq \frac{1}{\omega_Q}\left\{\int_{\mathbb{H}^n}\left(\int_{|\xi_j|_h=1}\left|f\left(\delta_{|x|_h} \xi_j\right)\right|^{p_i} d \xi_j\right)\left(\int_{|\xi_j|_h=1} d \xi_j\right)^{{p_i} / p^{\prime}_i} |x|_h^{\frac{p_i\alpha_i}{p}}\mathrm{d} x\right\}^{1 / {p_i}} \\
& =\omega_Q^{-1 / {p_i}}\left\{\int_0^{+\infty} \int_{|x^{\prime}|_h=1}\left(\int_{|\xi_j|_h=1}|f(\delta_r \xi_j)|^{p_i}d \xi_j\right) r^{\frac{p_i\alpha_i}{p}+Q-1} d x^{\prime} d r\right\}^{1 / {p_i}} \\
& =\left(\int_{\mathbb{H}^n}|f_j(y)|^{p_i} d y\right)^{1 / {p_i}} \\
& =\|f_j\|_{L^{p_i}(|x|_h^{\frac{p_i\alpha_i}{p}}dx)}.
\end{aligned}
$$
Thus, we have obtained
$$
\frac{\|H(f_1,\ldots,f_m)\|_{L^p(|x|_h^{\alpha}dx)}}{\prod_{j=1}^m\|f_j\|_{L^{p_j}(|x|_h^{\frac{p_i\alpha_i}{p}}dx)}}
\leq\frac{\|H(g_1,\ldots,g_m)\|_{L^p(|x|_h^{\alpha}dx)}}{\prod_{j=1}^m\|g_j\|_{L^{p_j}(|x|_h^{\frac{p_i\alpha_i}{p}}dx)}}.
$$
This implies the operator $H$ and its restriction to radial function have same norm in $L^p(\mathbb{H}^n)$, without loss of generality, we assume that $f_j$ for all $ j=1,\ldots,m$ are radial functions in the rest of the proof.

By Minkowski's inequality, H\"{o}lder's inequality and Lemma \ref{main_11}, we have
$$
\begin{aligned}
\|H(f_1,\ldots,f_m)\|_{L^p(|x|_h^{\alpha}dx)}
&=\left(\int_{\mathbb{H}^n}\left|\int_{\mathbb{H}^{nm}}K(e_1,y_1,\ldots,y_m)\prod^m_{j=1}f_j(\delta_{|x|_h}y_j)dy_1\cdots dy_m\right|^p|x|_h^{\alpha}dx\right)^{1/p}\\
&\leq\int_{\mathbb{H}^{nm}}K(e_1, y_1,\ldots , y_m)\left(\int_{\mathbb{H}^n}\left|\prod_{j=1}^{m}f_j(\delta_{|y_j|_h}x)\right|^p|x|_h^{\alpha}dx\right)^{1/p}dy_1\cdots dy_m\\
&\leq\int_{\mathbb{H}^{nm}}K(e_1, y_1,\ldots , y_m)\prod_{j=1}^{m}\left(\int_{\mathbb{H}^n}|f_j(\delta_{|y_j|_h}x)|^{p_j}|x|_h^{p_j\alpha_j/p}dx\right)^{1/p_j}dy_1\cdots dy_m\\
&=\int_{\mathbb{H}^{nm}}K(e_1, y_1,\ldots , y_m)\prod_{j=1}^{m}\|f(\delta_{|y|_h}\cdot)\|_{L^{p_j}(|x|_h^{p_j\alpha_j/p}dx)}dy_1\cdots dy_m.
\end{aligned}
$$
Using Lemma \ref{main_10}, we have
$$
\begin{aligned}
\|H(f_1,\ldots,f_m)\|_{L^p(|x|_h^{\alpha}dx)}
&\leq\int_{\mathbb{H}^{n m}} K(e_1,y_1,\ldots,y_m)\prod_{j=1}^m|y_j|_h^{-Q/p_j-\alpha_j/p}dy_1 \cdots dy_m\prod_{j=1}^m \|f_j\|_{L^{p_j}(|x|_h^{p_j\alpha_j/p}dx)}\\
&\leq D_m \prod_{j=1}^m\|f_j\|_{L^{p_j}(|x|_h^{p_j\alpha_j/p}dx)}.
\end{aligned}
$$
Next, take
$$
f_j=|x|_h^{-\left(\frac{Q}{p^j}+\frac{\alpha_j}{p}\right)},
$$
through a simple calculation, we can get
$$
H(f_1,\ldots,f_m)=D_m|x|_h^{-\left(\frac{Q}{p}+\alpha\right)}.
$$
Moreover,
$$
\|H(f_1,\ldots,f_m)\|_{L^{p}(|x|_h^\alpha dx)}=D_m\prod_{j=1}^m\|f_j\|_{L^{p^j}(|x|_h^{\frac{p^j\alpha^j}{p}}dx)}.
$$
The proof of Theorem $\ref{61}$ is finished.
\end{proof}
\section{Some special cases}
In this section, we will give some special cases for the multilinear operators on Heisenberg group. The Hilbert operator is the essential extension of the classical Hilbert's inequality , the multilinear Hilbert operator is defined by
$$
B_m(f_1,\ldots, y_m)(x):=\int_{\mathbb{R}^{nm}}\frac{f_1(y_1)\cdots f_m(y_m)}{(|x|^n+|y_1|^n+\cdots+|y_m|^n)^m}d y_1\cdots dy_m, x\in\mathbb{R}^n\backslash\{0\}.
$$
For $n=1$, the following is a know sharp estimate
$$
\int_0^{\infty} B_1^* f(x) g(x) d x \leq \frac{\pi}{\sin (\pi / p)}\|f\|_{L^p(0, \infty)}\|g\|_{L^{p^{\prime}(0, \infty)}}.
$$
Now, we give the definition of the multilinear Hilbert operator on Heisenberg group. 
\begin{defin}
	Suppose that each $f_i$ is a measurable function on $\mathbb{H}^n$ for $i=1,\ldots,m$. The mulitilinear Hilbert operator $B$ on Heisenberg group is defined by
	\begin{equation}\label{hilbert}
		B_m^h(f_1,\ldots,f_m)(x):=\int_{\mathbb{H}^n}\cdots\int_{\mathbb{H}^n}\frac{\prod_{i=1}^m f_i(x_i)}{(|x|_h^Q+|y_1|_h^Q+\cdots+|y_m|_h^Q)^m}dy_1\cdots dy_m.
	\end{equation} 
\end{defin}
\begin{thm}
	If $1<p_1,\ldots,p_m\leq\infty$ satisfy
	$$
	0<\frac{1}{p_1}+\cdots+\frac{1}{p_m}=\frac{1}{p}\leq1,
	$$
	then $B_m^h$ is bounded from $L^{p_1}(\mathbb{H}^n)\times\cdots\times L^{p_m}(\mathbb{H}^n)$ to $L^p(\mathbb{H}^n)$ and 
	$$
	\|B_m^h\|_{L^{p_1}(\mathbb{H}^n)\times \cdots\times L^p(\mathbb{H}^n)\rightarrow L^p(\mathbb{H}^n)}=\left(\frac{2 \pi^{n+\frac{1}{2}} \Gamma(n / 2)}{(n+1) \Gamma(n) \Gamma((n+1)/2)}\right)^m\frac{\prod^m_{i=1}\Gamma(1-\frac{1}{p_i})\Gamma(\frac{1}{p})}{\Gamma(m)}.
	$$
\end{thm}
\begin{proof}
	The kernel of the operator in (\ref{hilbert}) satisfies the conditions of Theorem \ref{main_6}. Using the polar coordinates and making change of variables, we have
	$$
	\begin{aligned}
		\|B_m^h\|_{\prod_{i=1}^m L^{p_i}(\mathbb{H}^n)\rightarrow L^p(\mathbb{H}^n)}&=\int_{\mathbb{H}^{nm}}\frac{|y_1|_h^{-Q/p_1}\cdots|y_m|_h^{-Q/p_m}}{(1+|y_1|_h^Q+\cdots+|y_m|_h^Q)^m} dy_1\cdots dy_m\\
		&=\omega_Q^m\int_0^\infty\cdots\int_0^\infty\frac{r_1^{Q-Q/p_1-1}\cdots r_m^{Q-Q/p_m-1}}{(1+r_1^Q+\cdots+r_m^Q)^m}dr_1\cdots dr_m\\
		&=\frac{\omega_Q^m}{Q^m}\int_{0}^{\infty}\cdots\int_{0}^{\infty}\frac{t_1^{-1/p_1}\cdots t_m^{-1/p_m}}{(1+t_1+t_2+\cdots+t_m)^m}dt_1\cdots dt_m,
	\end{aligned}
	$$
	Set
	$$
	I_m(\alpha,\beta_1,\ldots,\beta_m):=\int_0^\infty\cdots \int_0^\infty\frac{t_1^{-\beta_1}\cdots t_m^{-\beta_m}}{(1+t_1+\cdots+t_m)^\alpha}dt_1\cdots dt_m.
	$$
	By a simple calculate
	$$
	\int_0^\infty\frac{1}{(1+t)^\alpha t^\beta}dt=\int_0^\infty(1-t)^{-\beta}t^{\alpha+\beta-2}dt.
	$$
	By the definition of Beta function $B$, we have 
	$$
	\int_0^\infty\frac{1}{(1+t)^\alpha t^\beta}dt=B(1-\beta,\alpha+\beta-1).
	$$
	Using the variable substitution $t_m=(1+t_1+\cdots+t_m-1)q_m$, we have
	$$
	\begin{aligned}
		&I_m(\alpha,\beta_1,\ldots,\beta_m)\\
		&=\int_0^\infty\cdots\int_0^\infty\frac{t_1^{-\beta_1}\cdots t_{m-1}^{-\beta_{m-1}}}{(1+t_1+\cdots+t_{m-1})^{\alpha-1+\beta_m}}dt_1\cdots dt_{m-1}\int_0^\infty\frac{1}{(1+q_m)^\alpha q_m^{\beta_m}}d q_m\\
		&=B(1-\beta_m,\alpha+\beta_m-1)I_{m-1}(\alpha-1+\beta_m,\beta_1,\ldots,\beta_{m-1}).
	\end{aligned}
	$$
	Using the inductive method together with the properties of Gamma function $\Gamma$, we have
	$$
	I_m(\alpha,\beta_1,\ldots,\beta_m)=\frac{\prod_{i=1}^k \Gamma(1-\beta_i) \Gamma(\alpha-k+\prod_{i=1}^k \beta_i)}{\Gamma(\alpha)}.
	$$
	Thus, we can obtain
	$$
	\begin{aligned}
		\|B_m^h\|_{\prod_{i=1}^m H^\infty_{\alpha_i}(\mathbb{H}^n)\rightarrow H^\infty_\alpha(\mathbb{H}^n)}&=\frac{\omega_Q^m}{Q^m}\frac{\prod^m_{i=1}\Gamma(1-\frac{1}{p_i})\Gamma(\frac{1}{p})}{\Gamma(m)}\\
		&=\left(\frac{2 \pi^{n+\frac{1}{2}} \Gamma(n / 2)}{(n+1) \Gamma(n) \Gamma((n+1)/2)}\right)^m\frac{\prod^m_{i=1}\Gamma(1-\frac{1}{p_i})\Gamma(\frac{1}{p})}{\Gamma(m)}.
	\end{aligned}
	$$
	Thus, we finished the proof.
\end{proof}
The mutilinear Hardy-Littlewood-P\'{o}lya operator is defined by
$$
P_m(f_1,\ldots , f_m)(x):=\int_{\mathbb{R}^{nm}}\frac{f_1(y_1)\cdots f_m(y_m)}{[\text{max}(|x|^n,|y_1|^n,\ldots, |y_m|^n)]^m}dy_1\cdots dy_m, x\in\mathbb{R}^n\backslash\{0\}.
$$
When $m=1$, the linear Hardy-Littlewood-P\'{o}lya operator $P_1$ is considered in \cite{Hardy}, Hardy et al. obtained the norm of Hardy-Littlewood-P\'{o}lya operator on $L^p(\mathbb{R}^+)(1<q<\infty)$, that is
$$
\|P_1\|_{L^p(\mathbb{R}^+\rightarrow L^p(\mathbb{R}^+))}=\frac{p^2}{p-1}.
$$
We give the definition of the multilinear Hardy-Littlewood-P\'{o}lya operator on Heisenberg group.
\begin{defin}
	Suppose that $f_1,\ldots,f_m$ be  nonnegative locally integrable functions on $\mathbb{H}^n$. The $m$-linear $n$-dimensional Hardy-Littlewood-P\'{o}lya operator is defined by
	\begin{equation}\label{Hardyl}
		P^h_m(f_1,\ldots,f_m)(x)=\int_{\mathbb{H}^{nm}}\frac{f_1(y_1)\cdots f_m(y_m)}{[\text{max} (|x|_h^Q,|y_1|_h^Q,\ldots,|y_m|_h^Q)]^m}dy_1\cdots dy_m,  x\in\mathbb{H}^n\backslash\{0\}.
	\end{equation}
\end{defin} 
\begin{thm}
	If $1<p_1,\ldots,p_m\leq\infty$ satisfy
	$$
	0<\frac{1}{p_1}+\cdots+\frac{1}{p_m}=\frac{1}{p}\leq1,
	$$
	then $P_m^h$ is bounded from $L^{p_1}(\mathbb{H}^n)\times\cdots\times L^{p_m}(\mathbb{H}^n)$ to $L^p(\mathbb{H}^n)$ and 
	$$
	\|P_m^h\|_{L^{p_1}(\mathbb{H}^n)\times \cdots\times L^p(\mathbb{H}^n)\rightarrow L^p(\mathbb{H}^n)}=\frac{m\Omega_Q^mp}{\prod_{j=1}^m(1-1/p_j)}.
	$$
\end{thm}
\begin{proof}
	The kernel of the operator in (\ref{Hardyl}) satisfies the conditions of Theorem \ref{main_6}. In order to get the sharp constant of mutilinear Hardy-Littlewood-P\'{o}lya operator, we will consider the following two scenarios and combine them to obtain our result.
	
	\textbf{1.case when $m=2$.}
	In this case , we have
	$$
	P^h_2=\int_{\mathbb{H}^n}\int_{\mathbb{H}^n}\frac{|y_1|_h^{-Q/p_1}|y_2|_h^{-Q/p_2}}{[\text{max}(1,|y_1|_h^Q,|y_2|_h^Q)]^2}dy_1dy_2.
	$$
	By calculation,
	$$
	\begin{aligned}
		\int_{\mathbb{H}^n}\int_{\mathbb{H}^n}&\frac{|y_1|_h^{-Q/p_1}|y_2|_h^{-Q/p_2}}{[\text{max}(1,|y_1|_h^Q,|y_2|_h^Q)]^2}dy_1dy_2\\
		&=\int_{|y_1|_h<1}\int_{|y_2|_h<1}|y_1|_h^{-Q/p_1}|y_2|_h^{-Q/p_2}dy_1dy_2\\
		&+\int_{|y_1|_h>1}\int_{|y_2|_h\leq|y_1|_h}|y_1|_h^{-Q/p_1-2Q}|y_2|_h^{-Q/p_2}dy_1 dy_2\\
		&+\int_{|y_2|_h>1}\int_{|y_1|_h<|y_2|_h}|y_1|_h^{-Q/p_1}|y_2|_h^{-Q/p_2-2Q}dy_1dy_2\\
		&:=I_0+I_1+I_2.
	\end{aligned}
	$$
	$$
	\begin{aligned}
		I_0&=\int_{|y_1|_h<1}\int_{|y_2|_h<1}|y_1|_h^{-Q/p_1}|y_2|_h^{-Q/p_2}dy_1dy_2\\
		&=\frac{\Omega_Q^2}{(1-1/p_1)(1-1/p_2)}\\
		I_1&=\int_{|y_1|_h>1}\int_{|y_2|_h\leq|y_1|_h}|y_1|_h^{-Q/p_1-2Q}|y_2|_h^{-Q/p_2}dy_1 dy_2\\
		&=\frac{\Omega_Q}{1-1/p_2}\int_{|y_1|_h>1}|y_1|_h^{-Q/p-Q}dy_1\\
		&=\frac{\Omega_Q^2p}{(1-1/p_2)}.
	\end{aligned}
	$$
	Similarly,
	$$
	\begin{aligned}
		I_2&=\int_{|y_2|_h>1}\int_{|y_1|_h<|y_2|_h}|y_1|_h^{-Q/p_1}|y_2|_h^{-Q/p_2-2Q}dy_1dy_2\\
		&=\frac{\Omega_Q}{1-1/p_1}\int_{|y_2|_h>1}|y_2|_h^{-Q/p-Q}dy_2\\
		&=\frac{\Omega_Q^2p}{(1-1/p_1)}.
	\end{aligned}
	$$
	Thus,
	$$
	\begin{aligned}
		&\int_{\mathbb{H}^n}\int_{\mathbb{H}^n}\frac{|y_1|_h^{-Q/p_1}|y_2|_h^{-Q/p_2}}{[\text{max}(1,|y_1|_h^Q,|y_2|_h^Q)]^2}dy_1dy_2\\
		&=I_0+I_1+I_2\\
		&=\frac{2p\Omega_Q^2}{(1-1/p_1)(1-p_2)}.
	\end{aligned}\\
	$$
	\textbf{2.case when $m\geq3$.}
	
	Let
	$$
	\begin{aligned}
		&E_0=\{(y_1,\ldots,y_m)\in\mathbb{H}^n\times\cdots\times\mathbb{H}^n:|y_k|_h\leq1,1\leq k\leq m\};\\
		&E_1=\{(y_1,\ldots,y_m)\in\mathbb{H}^n\times\cdots\times\mathbb{H}^n:|y_1|_h>1,|y_k|_h\leq|y_1|_h,2\leq k\leq m\};\\
		&E_i=\{(y_1, \ldots, y_m) \in \mathbb{H}^n \times \cdots \times \mathbb{H}^n:|y_i|_h>1,|y_j|_h<|y_k|_h,|y_k|_h \leq|y_i|_h, 1 \leq j<i<k \leq m\};\\
		&E_m=\{(y_1,\ldots,y_m)\in\mathbb{H}^n\times\cdots\times\mathbb{H}^n:|y_m|_h>1,|y_j|_h<|y_m|_h,1<j<m\};
	\end{aligned}
	$$
	obviously, we have
	$$
	\bigcup_{j=0}^m E_j=\mathbb{H}^n \times \cdots \times \mathbb{H}^n,E_i \cap E_j=\emptyset.
	$$
	Taking
	$$
	K_j=\int_{\mathbb{H}^{nm}}\frac{\prod_{i=1}^m |y_i|_h^{-Q/p_i}}{[ \text{max}  (1,|y_1|_h^Q,\ldots,|y_m|_h^Q)]^m}dy_1\cdots dy_m ,
	$$
	then we begin to calculate $K_j$ with $j=0,1,\ldots,m$.
	$$
	K_0=\prod_{j=1}^m\int_{|y_j|\leq 1}|y_j|_h^{-Q/p_j}d y_j=\frac{\Omega_Q^m}{\prod_{j=1}^m(1-1/p_j)},
	$$
	$$
	\begin{aligned}
		K_1&=\int_{|y_1|>1}|y_1|_h^{-Q/p_1-Qm}dy_1\prod_{j=2}^m\int_{|y_j|_h\leq|y_1|_h}|y_j|_h^{-Q/p_j}dy_j\\
		&=\frac{\Omega_Q^{m-1}}{\prod_{j=2}^m(1-1/p_j)}\int_{|y_1|_h>1}|y_1|_h^{-Q/p-Q}dy_1\\
		&=\frac{\Omega_Q^mp}{\prod_{j=2}^m(1-1/p_j)}.
	\end{aligned}
	$$
	So we can deduce that
	$$
	K_j=\frac{\Omega_Q^mp}{ \prod_{1\leq i\leq m,i\neq j} (1-1/p_j)}.
	$$
	Then we can obtain that
	$$
	K_m=\frac{m\Omega_Q^mp}{\prod_{j=1}^m(1-1/p_j)}.
	$$
	Combining the above two cases, we finished the proof. 
\end{proof}
Next, define the mutilinear Hardy operator on Heisenberg group as follows,
\begin{equation}\label{Hardy1}
D_m^h(f_1,\ldots,f_m)(x):=\prod_{i=1}^m\frac{1}{|B(0,|x|_h)|}\int_{B(o,|x|_h)}|f_i(x_i)|dx_i,
\end{equation}
where $x, y_1,\ldots,y_m\in\mathbb{H}^n$. We conclude from (\ref{Hardy1}) that
$$
D_m^h(f_1,\ldots,f_m)(x)\prod_{i=1}^{m}Df_i(x),
$$
where
\begin{equation}\label{Hardy2}
	D^hf_i(x):=\frac{1}{|B(0,|x|_h)|}\int_{B(o,|x|_h)}|f_i(x_i)|dx_i.
\end{equation}
The operator $D^h$ in (\ref{Hardy2}) is defined by Wu and Fu in \cite{Wu}.

Now we define a kernel function $K$ as
\begin{equation}\label{Hardy3}
K(x, y_1,\ldots, y_m)=\frac{1}{(\Omega_Q|x|_h^Q)^m} \prod_{i=1}^m \chi_{|y_i|_h \leqslant|x|_h}(y_i, x).
\end{equation}
Clearly, the operator $D_m^h$ is equivalent to
$$
D_m^h(f_1, \ldots, f_m)(x)=\int_{\mathbb{H}^{n m}} K(x, y_1,\ldots, y_m) \prod_{i=1}^m f_i(x_i) d x_1 \cdots d x_m.
$$
Next, we formulate our main theorem.
\begin{thm}
	Suppose that $f_i \in L^{p_i}(\mathbb{H}^n)$ for $1<p_i \leqslant \infty$ with $i=1, \ldots, m$. If $1<p \leqslant \infty$ and
	$$
	\frac{1}{p_1}+\frac{1}{p_2}+\cdots+\frac{1}{p_m}=\frac{1}{p},
	$$
	then we have
	$$
	\|D_m^h(f_1, \ldots, f_m)\|_{L^p(\mathbb{H}^n)} \leqslant\left(\prod_{i=1}^m \frac{p_i}{p_i-1}\right) \prod_{i=1}^m\|f\|_{L^{p_i}(\mathbb{H}^n)}
	$$
	and $\prod_{i=1}^m \frac{p_i}{p_i-1}$ is the sharp bound.
	\end{thm}
	\begin{proof}
	The kernel in (\ref{Hardy3}) satisfies the Theorem \ref{main_6}. According to Theorem \ref{main_6}, we conclude that
	$$
	\begin{aligned}
		\|D_m^h\|_{L^{p_1}(\mathbb{H}^n) \times \cdots \times L^{p_m}(\mathbb{H}^n) \rightarrow L^p(\mathbb{H}^n)} & =\int_{\mathbb{H}^{n m}} K\left(e_1, y_1,\ldots, y_m\right) \prod_{i=1}^m|y_i|_h^{-\frac{Q}{p_i}} d y_1 \cdots d y_m \\
		& =\Omega_Q^{-m} \prod_{i=1}^m \int_{|y_i|_h \leqslant 1}|y_i|^{-\frac{Q}{p_i}} d y_i \\
		& =\Omega_Q^{-m} \prod_{i=1}^m \int_0^1 r_i^{Q-1} \int_{S^{Q-1}} r_i^{-\frac{Q}{p_i}} d \sigma\left(y_i^{\prime}\right) d r_i \\
		& =\left(\frac{\omega_{Q}}{\Omega_Q}\right)^m \prod_{i=1}^k \int_0^1 r_i^{Q-1-\frac{Q}{p_i}} d \sigma\left(y_i^{\prime}\right) d r_i \\
		& =\prod_{i=1}^m \frac{p_i}{p_i-1}
	\end{aligned}
	$$
	This completes the proof.
	\end{proof}
	We remark that we can use the H$\ddot{o}$lder's inequality to prove the $L^{p_1}(\mathbb{H}^n) \times \cdots \times L^{p_m}(\mathbb{H}^n) \rightarrow L^p(\mathbb{H}^n)$-boundedness of $D_m^h$, i.e.,
	$$
	\|D_m^h(f_1, \ldots, f_m)\|_{L^p}(\mathbb{H}^n) \leqslant \prod_{i=1}^m\|D^h f_i\|_{L^{p_i}(\mathbb{H}^n)} \leqslant\left(\prod_{i=1}^m \frac{p_i}{p_i-1}\right) \prod_{i=1}^m\|f_i\|_{L^{p i}(\mathbb{H}^n)}
	$$
	holds. Obviously $\prod_{i=1}^m \frac{p_i}{p_i-1}$ is a bound of the operator $D_m^h$. However, it is not clear that $\prod_{i=1}^m \frac{p_i}{p_i-1}$ must be the sharp bound.
\section{Sharp constants for multilinear integral operator on Morrey space with two power weight in Heisenberg group}
The  Morrey space with power  weight as a special case of Lebesgue space and $A_p$ weighted Morrey space have some very interesting properties. Inspired by \cite{He}, we begin to consider sharp constant for multilinear integral operator on Morrey space with power  weight in Heisenberg group. 

For any measurable function $\omega^*$ over a set $E$ is given by
$$
\omega^*(E)=\int_E w^* d x.
$$
In what follows, $B(|x|_h,R)$ denotes the ball centered at $|x|_h$ with radius $R$, $|B(|x|_h , R)|$ denotes the Lebesgue measure of $B(|x|_h, R)$.

We will use this notation in the following definition of two power weighted Morrey spaces.
\begin{defin}
	Let $\omega^*_1,\omega^*_2:\mathbb{H}^n\rightarrow(0,\infty)$ are positive  measurable function, $1 \leq q<\infty $ and $-1 / q \leq \lambda<0$. The two power weighted Morrey space $L^{q,\lambda}(\mathbb{H}^n,\omega^*_1,\omega^*_2)$ is defined by
	$$
	L^{q, \lambda}(\mathbb{H}^n, \omega^*_1, \omega^*_2)=\{f:\|f\|_{L^{q, \lambda}(\mathbb{H}^n, \omega^*_1, \omega^*_2)}<\infty\},
	$$
	where
	$$
	\|f\|_{L^{q, \lambda}\left(\mathbb{H}^n, w^*_1, w^*_2\right)}=\sup _{a \in \mathbb{H}^n, R>0} w_1(B(|a|_h, R))^{-(\lambda+1 / q)}\left(\int_{B(|a|_h, R)}|f(x)|^q w_2(x) d x\right)^{1 / q} .
	$$
\end{defin}
\begin{thm}\label{main_100}
	Let $f_i$ be radial functions in $L^{q_j,\lambda}(\mathbb{H}^n,|x|_h^\alpha,|x|^\frac{q_j\gamma_j}{q})$, $1\leq q<\infty$, $-\frac{1}{q}\leq\lambda<0$, $1<q_j<\infty$, $\frac{1}{q}=\frac{1}{q_1}+\cdots+\frac{1}{q_m}$, $\gamma=\gamma_1\cdots+\gamma_m$, $-\frac{1}{q_j} \leq \lambda_j<0$ with $j=1, \ldots, m$.
	Then we have
	\begin{equation}
		\|H(f_1,\ldots,f_m)\|_{L^{q,\lambda}(\mathbb{H}^n,|x|_h^\alpha,|x|_h^\gamma)}\leq A_m\prod_{j=1}^m\|f_j\|_{L^{q_j, \lambda}_j(\mathbb{H}^n,|x|_h^\alpha,|x|_h^{\frac{q_j \gamma_j}{q}})},
	\end{equation}
	where
	\begin{equation}
		A_m=\int_{\mathbb{H}^{nm}} K(e_1,y_1,\ldots,y_m)\prod_{j=1}^{m}|y_j|^{Q\lambda_j-\frac{1}{q_j}\frac{q_j\gamma_j}{q}+\alpha(\lambda_j+\frac{1}{q_j})} d y_1 \cdots d y_m.
	\end{equation}
	Moreover, if $\alpha\neq -Q,-\frac{1}{q_j}<\lambda_j<0$ and $q\lambda=q_j\lambda_j$ with $j=1,\ldots,m$, then we have
	\begin{equation}
		\|H(f_1,\ldots,f_m)\|_{\prod_{j=1}^m L^{q_j,\lambda_j}(\mathbb{H}^n,|x|_h^\alpha,|x|_h^{\frac{q_j\gamma_j}{q}})\rightarrow L^{q,\lambda}(\mathbb{H}^n,|x|_h^\alpha,|x|_h^\gamma)}= A_m.
	\end{equation}
\end{thm}
To facilitate the proof of Theorem \ref{main_100}, we need to give an important result that is similar in form of Lemma \ref{main_10}.
\begin{lemma}\label{main_101}
	Let $1\leq q<\infty$, $-\frac{1}{q}\leq\lambda<0$ and $\alpha,\gamma\in\mathbb{R}$. If $t\in\mathbb{H}^n$, $f\in L^{q,\lambda}(\mathbb{H}^n,|x|_h^\alpha,|x|_h^\gamma)$ , then we have
	\begin{equation}
		\|f(\delta_{|t|_h} \cdot)\|_{L^{q,\lambda}(\mathbb{H}^n,|x|_h^\alpha,|x|_h^\gamma)}=|t|_h^{Q\lambda-\frac{\gamma}{q}+\alpha(\lambda+\frac{1}{q})}\|f\|_{L^{q,\lambda}(\mathbb{H}^n,|x|_h^\alpha,|x|_h^\gamma)}.
	\end{equation}
	\begin{proof}
		$$
		\begin{aligned}
			& \|f(\delta_{|t|_h} \cdot)\|_{L^{q, \lambda}(\mathbb{H}^n,|x|_h^\alpha,|x|_h^{\gamma})} \\
			=& \sup _{a \in \mathbb{H}^n, R>0}\left(\int_{B(|a|_h, R)}|x|_h^\alpha d x\right)^{-(\lambda+\frac{1}{q})}\left(\int_{B(|a|_h, R)}|f(\delta_|t|_h x)|^q|x|_h^\gamma d x\right)^{\frac{1}{q}} \\
			=&  \sup _{a \in \mathbb{H}^n, R>0}\left(\int_{B(|a|_h, R)}|x|_h^\alpha d x\right)^{-(\lambda+\frac{1}{q})}\left(\int_{B( |a|_h, R)}|f(\delta _{|t|_h}x)|^q||t|_h x|^\gamma |t|_h^{-\gamma}d x\right)^{\frac{1}{q}} \\
			=&  |t|_h^{-\frac{Q}{q}-\frac{\gamma}{q}}\sup _{a \in \mathbb{H}^n, R>0}\left(\int_{B(|a|_h, R)}|x|_h^\alpha d x\right)^{-(\lambda+\frac{1}{q})}\left(\int_{B(|t\alpha|_h, |t|_h R)}|f(x)|^q|x|_h^\gamma d x\right)^{\frac{1}{q}} \\
			=&  |t|_h^{-\frac{Q}{q}-\frac{\gamma}{q}}\sup _{a \in \mathbb{H}^n, R>0}\left(\int_{B(|a|_h, R)}||t|_hx|_h^\alpha|t|_h^{-\alpha} d x\right)^{-(\lambda+\frac{1}{q})}\left(\int_{B(|t\alpha|_h, |t|_h R)}|f(x)|^q|x|_h^\gamma d x\right)^{\frac{1}{q}} \\
			=&  |t|_h^{Q\lambda-\frac{\gamma}{q}+\alpha(\lambda+\frac{1}{q})}\sup _{a \in \mathbb{H}^n, R>0}\left(\int_{B(|ta|_h, |t|_h R)}|x|_h^\alpha d x\right)^{-(\lambda+\frac{1}{q})}\left(\int_{B(|t\alpha|_h, |t|_h R)}|f(x)|^q|x|_h^\gamma d x\right)^{\frac{1}{q}}\\
			=& |t|_h^{Q\lambda-\frac{\gamma}{q}+\alpha(\lambda+\frac{1}{q})}\|f\|_{L^{q,\lambda}(\mathbb{H}^n,|x|_h^\alpha,|x|_h^\gamma)}.
		\end{aligned}
		$$
		This finishs the proof of Lemma \ref{main_101}.
	\end{proof}
\end{lemma}
Next, we will give the proof of Theorem \ref{main_100}.
\begin{proof}[Proof of Theorem $\ref{main_100}$]
	Set
	$$
	g_j(x)=\frac{1}{\omega_Q} \int_{|\xi_j|_h=1} f_j(\delta_{|x|_h} \xi_j) d\xi_j, \quad x \in \mathbb{H}^n,
	$$
	we have obtained that
	$$
	H(g_{f_1},\ldots,g_{f_m})(x)=H (f_1,\ldots,f_m)(x).
	$$
	Using Minkowski's inequality and  H\"{o}lder's inequality, for $j=1,\ldots,m$, we have
	$$
	\begin{aligned}
		&\|g_j\|_{L^{q_j,\lambda_j}(\mathbb{H}^n,|x|_h^\alpha,|x|_h^{\frac{q_j,\gamma_j}{q}})}\\
		=&\frac{1}{\omega_Q}\sup_{a\in\mathbb{H}^n,R>0}\left(\int_{B(|a|_h,R)}|x|_h^\alpha d x\right)^{-(\lambda+\frac{1}{q})}\left(\int_{B(|a|_h,R)}\left|\int_{|\xi_j|_h=1}f_j(\delta_{|x|_h}\xi_j)d \xi_j\right|^q|x|_h^{\frac{q_j\gamma_j}{q}}dx\right)^{\frac{1}{q}}\\
		\leq&\frac{1}{\omega_Q}\sup_{a\in\mathbb{H}^n,R>0}\left(\int_{B(|a|_h,R)}|x|_h^\alpha d x\right)^{-(\lambda+\frac{1}{q})}\int_{|\xi_j|_h=1}\left(\int_{B(|a|_h,R)}|f_j(\delta_{|x|_h}\xi_j)| ^q |x|_h^\frac{q_j\gamma_j}{q} d x \right)^{\frac{1}{q}} d \xi_j\\
		\leq&\left(\int_{B(|a|_h,R)}|x|_h^\alpha dx\right)^{-(\lambda+\frac{1}{q})}\left(\frac{1}{\omega_Q}\int_{|\xi_j|_h=1}\int_{B(|a|_h,R)}|f_j(\delta_{|x|_h}\xi_j)|^q |x|^{\frac{q_j\gamma_j}{q}}dx d \xi_j\right)^{\frac{1}{q}}\\
		=&\|f_j\|_{L^{q_j,\lambda_j}(\mathbb{H}^n,|x|_h^\alpha,|x|_h^{\frac{q_j\gamma_j}{q}})}.
	\end{aligned}
	$$
	Thus, we have
	$$
	\frac{\|H(f_1,\ldots,f_m)\|_{L^{q,\lambda}(\mathbb{H}^n,|x|_h^\alpha,|x|_h^\gamma)}}{\prod_{j=1}^m\|f_j\|_{L^{q_j,\lambda_j}(\mathbb{H}^n,|x|_h^\alpha,|x|_h^{\frac{q_j\gamma_j}{q}})}}
	\leq\frac{\|H(g_1,\ldots,g_m)\|_{L^{q,\lambda}(\mathbb{H}^n,|x|_h^\alpha,|x|_h^\gamma)}}{\prod_{j=1}^m\|g_j\|_{L^{q_j,\lambda_j}(\mathbb{H}^n|x|_h^\alpha,|x|_h^{\frac{q_j\gamma_j}{q}})}}.
	$$
	Then by Minkowski's inequality, H\"{o}lder's inequality, Lemma  \ref{main_11} and Lemma \ref{main_101}, we can easily get
	$$
	\begin{aligned}
		&\|H(f_1,\ldots,f_m)\|_{L^{q,\lambda}(\mathbb{H}^n,|x|_h^\alpha,|x|_h^\gamma)}\\
		\leq&\int_{\mathbb{H}^{nm}} K(e_1,y_1,\ldots,y_m)\prod_{j=1}^{m}|y_j|^{Q\lambda_j-\frac{1}{q_j}\frac{q_j\gamma_j}{q}+\alpha(\lambda_j+\frac{1}{q_j})} d y_1 \cdots d y_m\prod_{j=1}^m\|f_j\|_{L^{q_j, \lambda}(\mathbb{H}^n,|x|_h^\alpha,|x|_h^{\frac{q_j \gamma_j}{q}})}\\
		=&A_m\prod_{j=1}^m\|f_j\|_{L^{q_j, \lambda}(\mathbb{H}^n,|x|_h^\alpha,|x|_h^{\frac{q_j \gamma_j}{q}})}.
	\end{aligned}
	$$
	Taking
	$$
	f_j(x)=|x|_h^{Q\lambda_j-\frac{1}{q_j}\frac{q_j\gamma_j}{q}+\alpha(\lambda_j+\frac{1}{q_j})}, j=1,\ldots,m,
	$$
	we have
	$$
	\|H(f_1,\ldots,f_m)\|_{L^{q,\lambda}(\mathbb{H}^n,|x|_h^\alpha,|x|_h^\gamma)}=A_m\prod_{j=1}^m\|f_j\|_{L^{q_j, \lambda}(\mathbb{H}^n,|x|_h^\alpha,|x|_h^{\frac{q_j \gamma_j}{q}})}.
	$$
	The proof of Theorem \ref{main_100} is finished.
\end{proof}
\section{The Boundedness for multilinear Calder$\acute{o}$n-Zygmund operators on Heisenberg group weighted Morrey spaces}
In this section, we give the boundedness for multilinear integral operator on Heisenberg group weighted Morrey spaces.
The classical $A_p$ weight theory was introduced by Muckenhoupt in the study of weighted
Lebesgue boundedness of Hardy-Littlewood maximal functions, one can see Chapter 7 in \cite{Gra1}.

\begin{defin}
	A weight $\omega$ is a nonnegative locally integrable function on $\mathbb H^n$. We denote the ball
of radius $r$ centered $x_0$ by $B=B(x_0,r)$, we say that $\omega\in A_p$ for some $1<p<\infty$, if
	$$\left(\frac1{|B|}\int_B \omega(x)\,dx\right)\left(\frac1{|B|}\int_B \omega(x)^{-\frac{1}{p-1}}\,dx\right)^{p-1}\le C \quad\mbox{for every ball}\; B\subseteq \mathbb
	H^n,$$ where $C$ is a positive constant which is independent of $B$.\\
	We say $\omega\in A_1$, if
	$$\frac1{|B|}\int_B \omega(x)\,dx\le C\,\underset{x\in B}{\mbox{ess\,inf}}\,\omega(x)\quad\mbox{for every ball}\;B\subseteq\mathbb H^n,$$
	we denote $${A_\infty } = \bigcup\limits_{1 \le p < \infty } {{A_p}}.$$
\end{defin}
\begin{defin}
	A weight function $\omega$ is said to belong to the reverse H\"{o}lder class $RH_r$ if there exist two constants $r>1$ and $C>0$ such that the following reverse H\"{o}lder inequality holds
	$$\left(\frac{1}{|B|}\int_B \omega(x)^r\,dx\right)^{1/r}\le C\left(\frac{1}{|B|}\int_B \omega(x)\,dx\right)\quad\mbox{for every ball}\; B\subseteq \mathbb H^n.$$
	It is well known that if $\omega\in A_p$ with $1<p<\infty$, then $\omega\in A_r$ for all $r>p$ and $\omega\in A_q$ for any $1<q<p$. If $\omega\in A_p$ with $1\le p<\infty$, then there exists $r>1$ such that $\omega\in RH_r$.
\end{defin}

Now let us recall the definitions of multiple weights.
\begin{defin}
	For $m$ exponents $p_1,\ldots,p_m$, we denote $\vec{P}$ by the vector $\vec{P}=(p_1,\ldots,p_m)$. Let $p_1,\ldots,p_m\in[1,\infty)$ and $p\in(0,\infty)$ with $1/p=\sum_{k=1}^m 1/{p_k}$. Given $\vec{\omega}=(\omega_1,\ldots,\omega_m)$, set $\nu_{\vec{\omega}}=\prod_{i=1}^m \omega_i^{p/{p_i}}$. We say that $\vec{\omega}$ satisfies the $A_{\vec{P}}$ condition if it satisfies
	\begin{equation}
	\sup_B\left(\frac{1}{|B|}\int_B \nu_{\vec{\omega}}(x)\,dx\right)^{1/p}\prod_{i=1}^m\left(\frac{1}{|B|}\int_B \omega_i(x)^{1-p'_i}\,dx\right)^{1/{p'_i}}<\infty,
	\end{equation}
	where $p_i=1,$ $\left(\frac{1}{|B|}\int_B \omega_i(x)^{1-p'_i}\,dx\right)^{1/{p'_i}}$ is understood as ${(\mathop {\inf }\limits_{x \in B} {\omega _i}(x))^{ - 1}}$.
\end{defin}

The classical Morrey spaces $L^{p,\lambda}$ were first introduced by Morrey  \cite{Morrey} to study the local behavior of solutions to second order elliptic partial differential equations. In 2009, Komori and Shirai \cite{Komori} considered the weighted  Morrey spaces $L^{p,\kappa}(\omega)$ and studied the boundedness of some classical operators such as the Hardy-Littlewood maximal operator and the Calder\'on-Zygmund operator on these spaces.

\begin{defin}\label{main_3}
	Let $0<p<\infty$, $0<\kappa<1$ and $\omega$ be a weight function on $\mathbb H^n$, then the weighted Morrey space is defined by
	\begin{equation*}
	L^{p,\kappa}(\omega)=\big\{f:	\big\|f\big\|_{L^{p,\kappa}(\omega)}= \mathop {\sup }\limits_{B\subseteq\mathbb{H}^n} \omega {(B)^{ - \frac{\kappa }{p}}}{\| f \|_{{L^p}(B,\omega dx)}}<\infty\big\}.
	\end{equation*}
\end{defin}

\begin{defin}
	Let $0<p<\infty$, $0<\kappa<1$ and $\omega$ be a weight function on $\mathbb H^n$, then  weighted weak Morrey space is defined by
	\begin{equation}
W L^{p, \kappa}(\omega)=\left\{f:	\big\|f\big\|_{WL^{p,\kappa}(\omega)}= \mathop {\sup }\limits_{B\subseteq\mathbb{H}^n} \omega {(B)^{ - \frac{\kappa }{p}}}{\left\| f \right\|_{W{L^p}(B,\omega dx)}}<\infty\right\}.
\end{equation}
\end{defin}

We now recall the definitions of multilinear Calderón-Zygmund operators with Dini kernel.
\begin{defin}
	For any $t \in (0,\infty ),$ let ${K}(x,{y_1}, \cdots ,{y_m})$ be a locally integrable function defined away from the diagonal $x = {y_1} =  \cdots  = {y_m}$ in $(\mathbb H^n)^{m+1}$. We say $K$ is a kernel of type $\theta$ if for some constants $A>0,$ such that 
	\begin{enumerate}
		\item[\emph{(1)}]
		$\left| {K(x,\vec y)} \right| \le \frac{A}{{{{(\sum\limits_{j = 1}^m {{{\left| {{{({y_j})}^{ - 1}}x} \right|}_h}} )}^{mn}}}};$
		\item[\emph{(2)}]
		$\left| {K(x,\vec y) - K(x,{y_1}, \cdots ,{y_i}^\prime , \cdots ,{y_m})} \right| \le \frac{A}{{{{(\sum\limits_{j = 1}^m {{{\left| {{{({y_j})}^{ - 1}}x} \right|}_h}} )}^{mn}}}} \cdot \theta (\frac{{{{\left| {{{({{y'}_i})}^{ - 1}}{y_i}} \right|}_h}}}{{\sum\limits_{j = 1}^m {{{\left| {{{({y_j})}^{ - 1}}x} \right|}_h}} }});$
		\item[\emph{(3)}]
		$\left| {K(z,\vec y) - K(x,\vec y)} \right| \le \frac{A}{{{{(\sum\limits_{j = 1}^m {{{\left| {{{({y_j})}^{ - 1}}x} \right|}_h}} )}^{mn}}}} \cdot \theta (\frac{{{{\left| {{{(x)}^{ - 1}}z} \right|}_h}}}{{\sum\limits_{j = 1}^m {{{\left| {{{({y_j})}^{ - 1}}x} \right|}_h}} }}),$
	\end{enumerate}
	where $(2)$ holds for any $i \in \{ 1, \cdots ,m\}$, whenever ${\left| {{{({y_i}^\prime )}^{ - 1}}{y_i}} \right|_h} \le \frac{1}{2}\mathop {\max }\limits_{1 \le j \le m} \{ {\left| {{{({y_j})}^{ - 1}}x} \right|_h}\}$
	and $(3)$ holds whenever ${\left| {{{(x)}^{ - 1}}z} \right|_h} \le \frac{1}{2}\mathop {\max }\limits_{1 \le j \le m} \{ {\left| {{{({y_j})}^{ - 1}}x} \right|_h}\}$.
\end{defin}
When $\theta \left( t \right) = {t^\gamma }$ for some $\gamma  > 0$, we say $K$ is a $m$-linear Calder\'on-Zygmund kernel.

We say $T: {\mathscr S}({\hn}) \times  \cdots  \times {\mathscr S}({\hn}) \to {\mathscr S}'({\hn})$ is an $m$-linear Calder\'on-Zygmund operator with kernel $K$ if
\begin{equation*}\label{NW8}
T(\vec f)(x) = \int_{{\nm}} {K(x,\vec y)\prod\limits_{j = 1}^m {{f_j}({y_j})d{\vec y}}},
\end{equation*}
for any $\vec f \in {C_c^\infty}({\hn}) \times  \cdots  \times {C_c^\infty}({\hn})$ and any $x \notin \bigcap\limits_{j = 1}^m {{\rm{supp}}{f_j}}$, and $T$ can be extended to be a bounded operator from ${L^{{q_1}}} \times  \cdots \times{L^{{q_m}}}$ to ${L^q}$, for some $1 \le {q_1} \cdots , {q_m} < \infty, \frac{1}{q} = \sum\limits_{k = 1}^m {\frac{1}{{{q_k}}}}$.

$T$ is called a $m$-linear Calder\'on-Zygmund operator with Dini kernel $K$ when $K$ is a kernel of type $\theta \in Dini(1)$.

The following Lemma was proved in \cite{LuZhang}, but the same results can be obtained on Heisenberg group $\hn$ and the proof is similar to before.
\begin{lem}\label{C-Z}
	Let $m\in \n$ and $T$ be an $m$-linear Calder\'on-Zygmund operator with Dini kernel $K$. If $p_1,\ldots,p_m\in[1,\infty)$, $p\in(0,\infty)$ with $1/p=\sum_{k=1}^m 1/{p_k}$, and $\vec{\omega}=(\omega_1,\ldots,\omega_m)\in A_{\vec{P}}$, the following results hold:
	\begin{enumerate}[(i)]
		\item If $\mathop {\min }\limits_{1 \le i \le m} \{ {p_i}\}  > 1,$ then there exists a constant ${C}$, independent of $\vec f$, such that 
		\begin{equation*}
		{\left\| {T(\vec f)} \right\|_{{L^{p}}({v_{\vec \omega }})}} \le C{\prod\limits_{i = 1}^m {\left\| {{f_i}} \right\|} _{{L^{{p_i}}}({\omega _i})}};
		\end{equation*}
		\item If $\mathop {\min }\limits_{1 \le i \le m} \{ {p_i}\}  = 1,$ then there exists a constant ${C}$, independent of $\vec f$, such that 
		\begin{equation*}
		{\left\| {T(\vec f)} \right\|_{W{L^{p}}({v_{\vec \omega }})}} \le C{\prod\limits_{i = 1}^m {\left\| {{f_i}} \right\|} _{{L^{{p_i}}}({\omega _i})}}.
		\end{equation*}
	\end{enumerate}
\end{lem}
Next, we give the following results for multilinear integral operators.
\begin{thm}{\label{main_15}}
The $m$-linear integral operators with kernel $K$, which satisfies size condition
\begin{equation}
| {K(x,\vec y)} | \leq \frac{C}{\left(\sum\limits_{j = 1}^m {|(y_j)^{ - 1}x |_h} \right)^{mQ}}.
\end{equation}
If $1/p=\sum_{k=1}^m 1/{p_k}$ with $p_1,\ldots,p_m\in[1,\infty)$ and $\vec{\omega}=(\omega_1,\ldots,\omega_m)\in {A_{\vec P}} \cap {\left( {{A_\infty }}\right)^m}$. For any $0<\kappa<1$, the following results hold.
\begin{enumerate}[(i)]
	\item If $\mathop {\min }\limits_{1 \le i \le m} \{ {p_i}\}  > 1$, such that H is well-defined on ${L^{{p_1}}}({\omega _1}) \times  \cdots  \times {L^{{p_m}}}({\omega _m})$, which is also bounded from ${L^{{p_1}}}({\omega _1}) \times  \cdots  \times {L^{{p_m}}}({\omega _m})$ to ${L^p}({v_{\vec \omega }}),$ then we have
	\begin{equation}
	{\| {T(\vec f)} \|_{{L^{p,\kappa}}({v_{\vec \omega }})}} \lesssim {\prod\limits_{i = 1}^m {\| {{f_i}} \|} _{{L^{{p_i},\kappa }}({\omega _i})}}.
	\end{equation}
	\item If $\mathop {\min }\limits_{1 \le i \le m} \{ {p_i}\}  = 1,$ such that H is well-defined on ${L^{{p_1}}}({\omega _1}) \times  \cdots  \times {L^{{p_m}}}({\omega _m})$, which is also bounded from ${L^{{p_1}}}({\omega _1}) \times  \cdots  \times {L^{{p_m}}}({\omega _m})$ to $W{L^p}({v_{\vec \omega }}),$ then we have
	\begin{equation}
	{\| {T(\vec f)} \|_{W{L^{p,\kappa}}({v_{\vec \omega }})}} \lesssim {\prod\limits_{i = 1}^m {\| {{f_i}} \|} _{{L^{{p_i},\kappa}}({\omega _i})}}.
	\end{equation}
\end{enumerate}
\end{thm}
In order to prove Theorem \ref{main_15}, we need to give the following lemmas .
\begin{lemma}[\cite{Gra1}]\label{lem1}
	Let $\omega\in A_p$, $p\ge1$, for any ball $B \subseteq \mathbb{H}^n$, there exists a constant $C$ such that
	\begin{equation*}
	\omega(2B)\le C \omega(B).
	\end{equation*}
	In general, for any $\lambda>1$, we have
	\begin{equation}\label{LEQ1}
	\omega(\lambda B)\le C\lambda^{Qp}\omega(B),
	\end{equation}
	where $C$ is nether depend on $B$ nor depend on $\lambda$.
\end{lemma}

\begin{lem}[\cite{Gra1}]\label{lem3}
	Let $\omega\in RH_r$ with $r>1$, then there exists  a constant $C$ such that
	\begin{equation}\label{AAA5}
	\frac{\omega(E)}{\omega(B)}\le C\left(\frac{|E|}{|B|}\right)^{(r-1)/r}
	\end{equation}
	for any measurable subset $E$ of a ball $B$.
\end{lem}

\begin{lemma}[\cite{Lerner}]\label{lem4}
	Let $1/p=\sum_{k=1}^m 1/{p_k}$ with $p_1,\ldots,p_m\in[1,\infty)$, then $\vec{\omega}=(\omega_1,\ldots,\omega_m)\in A_{\vec{P}}$ if and only if
	\begin{equation*}\left\{
	\begin{aligned}
	&\nu_{\vec{\omega}}\in A_{mp},\\
	&\omega_i^{1-p'_i}\in A_{mp'_i},\quad i=1,\ldots,m,
	\end{aligned}\right.
	\end{equation*}
	where $\nu_{\vec{\omega}}=\prod_{i=1}^m \omega_i^{p/{p_i}}$ and the condition $\omega_i^{1-p'_i}\in A_{mp'_i}$ in the case $p_i=1$ is understood as $\omega_i^{1/m}\in A_1$.
\end{lemma}

\begin{lem}[\cite{Wang1}]\label{lem5}
	Let $m\in \n,$ $p_1,\ldots,p_m\in[1,\infty)$ and $1/p=\sum_{k=1}^m 1/{p_k}$ with $p\in(0,\infty)$. Assume that $\omega_1,\ldots,\omega_m\in A_\infty$ and $\nu_{\vec{\omega}}=\prod_{i=1}^m \omega_i^{p/{p_i}}$, then for any ball $B,$ we have
	\begin{equation}\label{N1}
	\prod\limits_{i = 1}^m {{{\left( {\int_B {{\omega _i}} (x){\mkern 1mu} dx} \right)}^{p/{p_i}}}}  \lesssim \int_B {{\nu _{\vec \omega }}} (x){\mkern 1mu} dx.
	\end{equation}
\end{lem}
Now, we begin to proof Theorem \ref{main_15}.
\begin{proof}[Proof of Theorem $\ref{main_15}$]
For any ball $B=B(x_0,r)$, let $f_i=f^0_i+f^{\infty}_i$, where $f^0_i=f_i\chi_{2B}$, for any $i=1,\ldots,m$ and $\chi_{2B}$ denotes the characteristic function of $2B$. Then we obtain
	\begin{equation}
	\begin{aligned} \prod_{i=1}^m f_i(y_i) & =\prod_{i=1}^m(f_i^0(y_i)+f_i^{\infty}(y_i)) \\ & =\sum_{\alpha_1, \ldots ,\alpha_m \in\{0, \infty\}} f_1^{\alpha_1}(y_1) \cdots f_m^{\alpha_m}(y_m) \\ & =\prod_{i=1}^m f_i^0(y_i)+\sum_{\alpha_1+\cdots+\alpha_m \neq 0} f_1^{\alpha_1}(y_1) \cdots f_m^{\alpha_m}(y_m).
   \end{aligned}
	\end{equation}
	Thus, we have
	\begin{equation}
		\begin{aligned}
			&{\nu _{\vec \omega }}{(B)^{ - \frac{\kappa }{p}}}{\| {H({f_1}, \ldots ,{f_m})} \|_{{L^p}(B,{\nu _{\vec \omega }}dx)}}\\
			\le &{\nu _{\vec \omega }}{(B)^{ - \frac{\kappa }{p}}}{\| {H(f_1^0, \ldots ,f_m^0)} \|_{{L^p}(B,{\nu _{\vec \omega }}dx)}} + \sum\limits_{{\alpha _1} +  \cdots  + {\alpha _m} \ne 0}^{} {{\nu _{\vec \omega }}{{(B)}^{ - \frac{\kappa }{p}}}{{\| {T(f_1^{{\alpha _1}}, \ldots ,f_m^{{\alpha _m}})} \|}_{_{{L^p}(B,{\nu _{\vec \omega }}dx)}}}} \\
			=: &{I^0} + \sum\limits_{{\alpha _1} +  \cdots  + {\alpha _m} \ne 0}^{} {{I^{{\alpha _1} \ldots ,{\alpha _m}}}}.
		\end{aligned}
	\end{equation}
	Next, we  need to prove
	\begin{equation}\label{N2}
		I^{\alpha_1, \ldots, \alpha_m} \lesssim \prod_{i=1}^m\|f_i\|_{L^{p_i, \kappa}(\omega_i)},
	\end{equation}
	where ${\alpha_i}\in(0,\infty) $, for any $i = 1, \ldots ,m$. Using Lemma \ref{lem4}, we have $\nu_{\vec{\omega}}\in A_{mp}$. Lemma \ref{lem1} and Lemma \ref{lem5} allows us to get
	\begin{equation}
		\begin{aligned}
I^0 & \lesssim \frac{1}{\nu_{\vec{\omega}}(B)^{\kappa / p}} \prod_{i=1}^m\left(\int_{2 B}|f_i(x)|^{p_i} \omega_i(x) d x\right)^{1 / p_i} \\
& \lesssim \prod_{i=1}^m\|f_i\|_{L^{p_i, \kappa}(\omega_i)} \cdot \frac{\prod_{i=1}^m \omega_i(2 B)^{\kappa / p_i}}{\nu_{\vec{\omega}}(B)^{\kappa / p}} \\
& \lesssim \prod_{i=1}^m\|f_i\|_{L^{p_i, \kappa(\omega_i)}} \cdot \frac{\nu_{\vec{\omega}}(2 B)^{\kappa / p}}{\nu_{\vec{\omega}}(B)^{\kappa / p}} \\
& \lesssim \prod_{i=1}^m\|f_i\|_{L^{p_i, \kappa}(\omega_i)} .
       \end{aligned}
	\end{equation}
	For $H$, we have
\begin{equation}
\begin{aligned}
|T(\vec{f})(x)| \lesssim \int_{\mathbb{H}^{nm}} \frac{|\prod_{j=1}^m f_j(y_j)|}{\left(\sum_{i=1}^m|(y_i)^{-1} x|_h\right)^{mQ}} d y_1 \cdots d y_m .
\end{aligned}
\end{equation}
	
To obtain these conclusions, we need establish some geometric relationships by trigonometric inequality as follows.\\
(i) If $x \in B$, $y \in(2 B)^c$, we  have
$$
\left|(y)^{-1} x\right|_h \approx\left|(y)^{-1} x_0\right|_h ;
$$
(ii) If $x \in B$, $y \in 2^{j+1} B \backslash 2^j B$, $j \in \mathbb{N}$, we  have
$$
|{(y)^{ - 1}}x|_h^Q \approx \left| {{2^{j + 1}}B} \right| .
$$
	For  other terms, we first consider the case of $\alpha_1=\cdots=\alpha_m=\infty$. For $x\in B$, we have
	\begin{equation}
\begin{aligned}
\left|H(f_1^{\infty}, \ldots, f_m^{\infty})(x)\right| & \lesssim \int_{\mathbb{H}^{nm} \backslash(2 B)^m} \frac{|f_1(y_1) \ldots f_m(y_m)|}{(\sum_{i=1}^m|(y_i)^{-1} x|_h)^{mQ}} d y_1 \cdots d y_m \\
& \leq \sum_{j=1}^{\infty} \int_{(2^{j+1} B)^m \backslash(2^j B)^m} \frac{|f_1(y_1) \cdots f_m(y_m)|}{(\sum_{i=1}^m|(y_i)^{-1} x|_h)^{mQ}} d y_1 \cdots d y_m \\
&\lesssim \sum_{j=1}^{\infty} \prod_{i=1}^m \frac{1}{|2^{j+1} B|} \int_{2^{j+1} B}|f_i(y_i)| d y_i .
\end{aligned}
   \end{equation}
	By using H\"older's inequality, $A_{\vec{P}}$ condition and Lemma \ref{lem5}, we obtain
	\begin{equation}
\begin{aligned}
|T(f_1^{\infty}, \ldots, f_m^{\infty})(x)| & \lesssim \sum_{j=1}^{\infty} \prod_{i=1}^m \frac{1}{|2^{j+1} B|}\left(\int_{2^{j+1} B}\left|f_i(y_i)\right|^{p_i} \omega_i(y_i) d y_i\right)^{1 / p_i}\left(\int_{2^{j+1} B} \omega_i(y_i)^{1-p_i^{\prime}} d y_i\right)^{1 / p_i^{\prime}} \\
& \lesssim \sum_{j=1}^{\infty} \frac{1}{|2^{j+1} B|^m} \cdot \frac{\left|2^{j+1} B\right|^{\frac{1}{p}+\sum_{i=1}^m\left(1-\frac{1}{p_i}\right)}}{\nu_{\vec{\omega}}(2^{j+1} B)^{1 / p}} \prod_{i=1}^m\left(\|f_i\|_{L^{p_i, \kappa}(\omega_i)} \omega_i(2^{j+1} B)^{\kappa / p_i}\right) \\
& =\left(\prod_{i=1}^m\|f_i\|_{L^{p_i, \kappa}(\omega_i)}\right) \cdot \sum_{j=1}^{\infty}\left(\frac{\prod_{i=1}^m \omega_i(2^{j+1} B)^{\kappa / p_i}}{\nu_{\vec{\omega}}(2^{j+1} B)^{1 / p}}\right) \\
& \lesssim\left(\prod_{i=1}^m\|f_i\|_{L^{p_i, \kappa}(\omega_i)}\right) \cdot \sum_{j=1}^{\infty} \nu_{\vec{\omega}}(2^{j+1} B)^{(\kappa-1) / p}.
\end{aligned}
\end{equation}
	Thus, by virtue of $\nu_{\vec{w}}\in A_{mp} \subseteq A_\infty$ and  Lemma \ref{lem3}, we have
	\begin{equation}
		\begin{aligned}
I^{\infty, \ldots, \infty} & \lesssim\left(\prod_{i=1}^m\left\|f_i\right\|_{L^{p_i, \kappa}\left(\omega_i\right)}\right) \cdot \sum_{j=1}^{\infty} \frac{\nu_{\vec{\omega}}(B)^{(1-\kappa) / p}}{\nu_{\vec{\omega}}\left(2^{j+1} B\right)^{(1-\kappa) / p}} \\
& \lesssim\left(\prod_{i=1}^m\left\|f_i\right\|_{L^{p_i, \kappa}\left(\omega_i\right)}\right) \cdot \sum_{j=1}^{\infty}\left(\frac{|B|}{\left|2^{j+1} B\right|}\right)^{\delta(1-\kappa) / p} \\
& \lesssim \prod_{i=1}^m\left\|f_i\right\|_{L^{p_i, \kappa}(\omega_i)},
\end{aligned}
	\end{equation}
	where we used the fact
	\begin{equation}\label{EQQ009}
		\frac{\nu_{\vec{w}}(B)}{\nu_{\vec{w}}(2^{j+1}B)}\lesssim\left(\frac{|B|}{|2^{j+1}B|}\right)^\delta.
	\end{equation}
	 The last inequality holds since $0<\kappa<1$ and $\delta>0$.\\
	
Without loss of generality, we may assume that ${\alpha _1} =  \cdots  = {\alpha _l} = \infty $ and ${\alpha _{l + 1}} =  \cdots  = {\alpha _m} = 0$. For any $x \in B,{y_1} \in {(2B)^c}$, we have
	\begin{equation}
    \begin{aligned}
|T(f_1^{\infty}, \ldots, f_{l}^{\infty}, f_{l+1}^0, \ldots, f_m^0)(x)|
\lesssim & \int_{(\mathbb{H}^n)^{l} \backslash(2 B)^{l}} \int_{(2 B)^{m-l}} \frac{|f_1(y_1) \cdots f_m(y_m)|}{(\sum_{i=1}^m|(y_i)^{-1} x|_h)^{mQ}} d y_1 \cdots d y_m \\
\lesssim& \left(\prod_{i=l+1}^m \int_{2 B}|f_i(y_i)| d y_i\right)\\
& \times \sum_{j=1}^{\infty} \frac{1}{|2^{j+1} B|^m} \int_{(2^{j+1} B)^{l} \backslash(2^j B)^{l}}|f_1(y_1) \ldots f_{l}(y_{l})| d y_1 \cdots d y_{l} \\
\leq & \sum_{j=1}^{\infty} \prod_{i=1}^m \frac{1}{|2^{j+1} B|} \int_{2^{j+1} B}|f_i(y_i)| d y_i,
\end{aligned}
\end{equation}
	where the second inequality is valid, then we have
$$
| {{2^{j + 1}}B} | \approx | {{{( {{y_1}} )}^{{\rm{ - }}1}}x} |_h^Q \le {\left(\sum\limits_{i = 1}^m {{{| {{{( {{y_i}} )}^{{\rm{ - }}1}}x} |}_h}} \right)^Q}.
$$
	It is the same situation as before. So for any $x\in B$, we also have
	\begin{equation}\label{EQQ012}
		\big|T(f^\infty_1,\ldots,f^\infty_l,f^0_{l+1},\ldots,f^0_m)(x)\big|\lesssim\left(\prod_{i=1}^m \big\|f_i\big\|_{L^{p_i,\kappa}(\omega_i)}\right)\cdot
		\sum_{j=1}^\infty\nu_{\vec{\omega}}\big(2^{j+1}B\big)^{{(\kappa-1)}/p}.
	\end{equation}
	Consequently, we obtain
	\begin{equation}
		\begin{aligned}
I^{\infty, \ldots \ldots, 0, \ldots, 0} & \leq \nu_{\vec{\omega}}(B)^{(1-\kappa) / p}\left|T(f_1^{\infty}, \ldots, f_{l}^{\infty}, f_{l+1}^0, \ldots, f_m^0)(x)\right| \\
& \lesssim\left(\prod_{i=1}^m\|f_i\|_{L^{p_i, \kappa}(\omega_i)}\right) \cdot \sum_{j=1}^{\infty} \frac{\nu_{\vec{\omega}}(B)^{(1-\kappa) / p}}{\nu_{\vec{\omega}}(2^{j+1} B)^{(1-\kappa) / p}} \\
& \lesssim\left(\prod_{i=1}^m\|f_i\|_{L^{p_i, \kappa}(\omega_i)}\right) \cdot \sum_{j=1}^{\infty}\left(\frac{|B|}{|2^{j+1} B|}\right)^{\delta(1-\kappa) / p} \\
& \lesssim \prod_{i=1}^m\|f_i\|_{L^{p_i, \kappa}(\omega_i)} .
\end{aligned}
	\end{equation}
	Combining with (\ref{N2}), we can finish this proof.
\end{proof}

\begin{rem}
If $T$ is an $m$-linear Calder\'on-Zygmund operator with Dini kernel $K$, we can easily to get the boundedness of $T$ as in Theorem $\ref{main_15}$ by Lemma $\ref{C-Z}$.
\end{rem}
\subsection*{Acknowledgements}
This work was supported by National Natural Science Foundation of  China (Grant No. 12271232) and Shandong Jianzhu University Foundation (Grant No. X20075Z0101).

	
\begin{flushleft}

	\vspace{0.3cm}\textsc{Xiang Li\\School of Science\\Shandong Jianzhu University\\Jinan, 250000\\P. R. China}
	
	\emph{E-mail address}: \textsf{lixiang162@mails.ucas.ac.cn}
	
	\vspace{0.3cm}\textsc{Xi Cen\\School of Mathematics and Science\\Southwest University of Science and Technology\\Mianyang, 621010\\P. R. China}
	
	\emph{E-mail address}: \textsf{xicenmath@gmail.com}
	
	\vspace{0.3cm}\textsc{Zunwei Fu\\Department of Mathmatics\\Linyi University \\Linyi, 270005\\P. R. China}
	
	\emph{E-mail address}: \textsf{fuzunwei@lyu.edu.cn}
	
	\vspace{0.3cm}\textsc{Zhongci Hang\\School of Science\\Shandong Jianzhu University \\Jinan, 250000\\P. R. China}
	
	\emph{E-mail address}: \textsf{babysbreath4fc4@163.com}

\end{flushleft}

\end{document}